\documentclass{siamart171218}
\usepackage{amsfonts,amsmath,amssymb,graphicx,caption}
\usepackage{graphicx,epstopdf}
\makeatletter
\newcommand*\rel@kern[1]{\kern#1\dimexpr\macc@kerna}
\newcommand*\widebar[1]{%
  \begingroup
  \def\mathaccent##1##2{%
    \rel@kern{0.8}%
    \overline{\rel@kern{-0.8}\macc@nucleus\rel@kern{0.2}}%
    \rel@kern{-0.2}%
  }%
  \macc@depth\@ne
  \let\math@bgroup\@empty \let\math@egroup\macc@set@skewchar
  \mathsurround\z@ \frozen@everymath{\mathgroup\macc@group\relax}%
  \macc@set@skewchar\relax
  \let\mathaccentV\macc@nested@a
  \macc@nested@a\relax111{#1}%
  \endgroup
}
\makeatother

\DeclareMathOperator*{\argmin}{argmin}
\usepackage{color}
\usepackage{MnSymbol}
\usepackage[ruled,algo2e]{algorithm2e}
\usepackage{amsmath,amssymb}
\usepackage{cite}
\usepackage{mathtools}
\usepackage{tikz}
\usepackage{pgfplots}
\usepackage{pgfplotstable}
\usepackage{geometry}
\usepackage{color}
\usepackage{booktabs}
\usepackage{framed}
\pgfplotsset{compat=1.9}

\pgfplotstableset{
	every head row/.style={before row=\toprule,after row=\midrule},
	clear infinite
}

\usepackage{amsopn}

\renewcommand{\leq}{\leqslant}
\renewcommand{\geq}{\geqslant}

\usepackage{pgfplots}			%	matlab2tikz
 \usepackage[caption=false]{subfig} %
\usepackage{float}

\usepackage{scalerel,stackengine}
\stackMath
\newcommand\reallywidehat[1]{%
\savestack{\tmpbox}{\stretchto{%
  \scaleto{%
    \scalerel*[\widthof{\ensuremath{#1}}]{\kern-.6pt\bigwedge\kern-.6pt}%
    {\rule[-\textheight/2]{1ex}{\textheight}}%WIDTH-LIMITED BIG WEDGE
  }{\textheight}%
}{0.5ex}}%
\stackon[1pt]{#1}{\tmpbox}%
}

\newcommand{\RR}{{\mathbb{R}}}

\usepackage{geometry}
\newtheorem{Ass}[theorem]{Assumption}

\newtheorem{num_example}[theorem]{Example}

\usepackage{multirow}

\title{On the Solution of the Nonsymmetric T-Riccati Equation\thanks{This work was supported by the Australian Research Council (ARC) Discovery Grant No.~DP1801038707}}

\author{Peter Benner\footnotemark[2] \and Davide Palitta\thanks{Department Computational Methods in Systems
and Control Theory (CSC),
Max Planck Institute for Dynamics of Complex Technical Systems, Magdeburg, Germany. E-mail:  \texttt{\{benner,palitta\}@mpi-magdeburg.mpg.de}}}

\begin{document}
\bibliographystyle{siam}
\maketitle

\begin{abstract}
The nonsymmetric T-Riccati equation is a quadratic matrix equation where the linear part corresponds to the so-called T-Sylvester or T-Lyapunov operator that has previously been studied in the literature. It has applications in macroeconomics and policy dynamics. So far, it presents
an unexplored problem in numerical analysis, and both, theoretical results and computational methods, are lacking in the literature.
In this paper we provide some sufficient conditions for the existence and uniqueness of a nonnegative minimal solution and its efficient computation is deeply analyzed.
Both the small-scale and the large-scale setting are addressed and Newton-Kleinman-like methods are derived. The convergence of these procedures to the minimal solution is proved and several numerical results illustrate the computational efficiency of the proposed methods.
\end{abstract}

%%%%%%%%%%%%%%%%%%%%%%%%%%%%%%%%%%%%%%%%%%%%%%%%%%%%%%%%%%%%%%%%%%%%%%%%%%%%%%%%%%%%%%%%%%%%%%%%%%%%%%%%%%%%%%%%%%%%%%%

\section{Introduction}\label{Introduction}
In this paper, we consider the nonsymmetric T-Riccati operator
\begin{equation*}
\mathcal{R}_T:\mathbb{R}^{n\times n}\rightarrow\mathbb{R}^{n\times n}, \quad \mathcal{R}_T(X):=DX+X^TA-X^TBX+C,
\end{equation*}
where $A,B,C,D\in\RR^{n\times n}$ and sufficient conditions for the existence and uniqueness of a minimal solution $X_{\min}\in\mathbb{R}^{n\times n}$ to
\begin{equation}\label{eq.TRiccati}
 \mathcal{R}_T(X)=0,
\end{equation}
are provided.

The solution of the nonsymmetric T-Riccati equation \eqref{eq.TRiccati} plays a role in solving
dynamics generalized equilibrium (DSGE) problems 
\cite{BinP97,SchU04,Sim01}. ``DSGE modeling is a method in macroeconomics that
attempts to explain economic phenomena, such as economic growth and
business cycles, and the effects of economic
policy''\footnote{\url{https://en.wikipedia.org/wiki/Dynamic_stochastic_general_equilibrium}}.
Equations of the form \eqref{eq.TRiccati} appear in certain procedures
for solving DSGE models using perturbation-based methods \cite{BinP97,Sim01}.

Taking inspiration from the (inexact) Newton-Kleinman method for standard algebraic Riccati equations,
we illustrate efficient numerical procedures for solving \eqref{eq.TRiccati}. Both the small-scale and the large-scale setting are addressed.
In particular, in the latter framework, we assume the matrices $A$ and $D$ to be such that the matrix-vector products $Av$ and $Dw$ require $\mathcal{O}(n)$ floating point operations (flops) for any $v,w\in\mathbb{R}^{n}$, and $B$ and $C$ low rank. These hypotheses remind us of the usual assumptions adopted when dealing with
large-scale standard algebraic Riccati equations. See, e.g.,  \cite{Benner2018,Benner2016,BennerSaak2013,Bini2012,Feitzinger2009,Heyouni2008/09,Jbilou2003,Lin2015,Simoncini2016a,Simoncini2014,Palitta2019} and the recent survey paper \cite{Benner2018a}. Indeed, in this context,
the solution $Z$ is numerically rank deficient \cite{Benner2016a} and low-rank approximations of the form $\widebar Z\widebar Z^T\approx Z$, $\widebar Z\in\mathbb{R}^{n\times t}$, $t\ll n$, are thus expected to be accurate.
We think that also in the case of the nonsymmetric T-Riccati equation it is possible to show that the singular values of the solution
$X$ to \eqref{eq.TRiccati} present a fast decay and low-rank approximations can thus be sought. This may be proved by combining the arguments in
\cite{Benner2016a} with bounds for the decay of the singular values of the solution of certain T-Sylvester equations \cite{Dopico2016}. However, this is beyond the scope of this paper and in section~\ref{Numerical examples} we restrict ourselves to illustrate how low-rank approximations turn out to be sufficiently accurate in the examples we tested.

The following is a synopsis of the paper. In section~\ref{Existence and uniqueness of a minimal solution} we present the result about the existence and uniqueness of a minimal solution $X_{\min}$ to \eqref{eq.TRiccati}. A Newton-Kleinman method for the computation of such a $X_{\min}$ is derived in section~\ref{The (inexact) Newton-Kleinman method} and its convergence features are proved in section~\ref{A convergence result}. The large-scale setting is addressed in section~\ref{The large-scale setting} where the convergence of an inexact Newton-Kleinman method equipped with a specific line search is illustrated. Some implementation details of the latter procedure are discussed in section~\ref{Implementation details}.
Several numerical results showing the effectiveness of the proposed approaches are reported in section~\ref{Numerical examples} while our conclusions are given in section~\ref{Conlusions}.

Throughout the paper we will adopt the following notation.
 The matrix inner product is
defined as $\langle X, Y \rangle_F ∶= \mbox{trace}(Y^T X)$ so that the induced norm is $\| X\|^2_F= \langle X, X\rangle_F$.
$I_n$ denotes the identity matrix of order $n$ and the subscript
is omitted whenever the dimension of $I$ is clear from the context.
The brackets $[\cdot]$ are used to concatenate matrices of conforming dimensions. In particular, a MATLAB-like notation is adopted and $[M,N]$ denotes the matrix obtained by augmenting $M$ with $N$.
 $A\geq0$ ($A>0$) indicates a nonnegative (positive) matrix, that is a matrix whose entries are all nonnegative (positive).
 Clearly, $A\leq0$ ($A<0$) if $-A\geq0$ ($-A>0$) and $A\geq B$ if $A-B\geq0$. Moreover, we recall that a matrix $A$
 is a Z-matrix if all its off-diagonal entries are nonpositive. It is easy to show that a Z-matrix can be written in the
 form $A=sI-N$ where $s\in\RR$ and $N\geq0$. If $s\geq\rho(N)$, where $\rho(\cdot)$ denotes the
 spectral radius, then $A$ is called M-matrix.

 Furthermore, we will always suppose that the following assumption holds.
\begin{Ass}\label{Ass1}
We assume that
 \begin{itemize}
  \item $B$ is nonnegative, $B\geq0$, and $C$ is nonpositive, $C\leq0$.
  \item $I\otimes D+(A^T\otimes I)\Pi$ is a nonsingular M-matrix where $\otimes$ denotes the Kronecker product while
  $\Pi\in\mathbb{R}^{n^2\times n^2}$ is a permutation matrix given by $\Pi:=\sum_{i=1}^n\sum_{j=1}^n E_{i,j}\otimes E_{j,i}$.
 \end{itemize}
\end{Ass}
The matrix $E_{i,j}\in\mathbb{R}^{n\times n}$ in Assumption~\ref{Ass1} is the matrix whose $(i,j)$-th entry is 1 while
all the others are zero.

Notice that $I\otimes D+(A^T\otimes I)\Pi$ being a nonsingular M-matrix implies that the T-Sylvester operator
$\mathcal{S}_T:\mathbb{R}^{n\times n}\rightarrow\mathbb{R}^{n\times n},$ $\mathcal{S}_T(X):=DX+X^TA$,
has a nonnegative inverse,
i.e., $\mathcal{S}^{-1}_T(X)\geq0$ for $X\geq0$. For the standard Sylvester operator
$\mathcal{S}:\mathbb{R}^{n\times n}\rightarrow\mathbb{R}^{n\times n},$ $\mathcal{S}(X):=DX+XA$, this is guaranteed by
assuming $A$, $D$ to be nonsingular M-matrices. See, e.g. \cite[Theorem A.20]{Bini2012}.

%%%%%%%%%%%%%%%%%%%%%%%%%%%%%%%%%%%%%%%%%%%%%%%
\section{Existence and uniqueness of a minimal solution}\label{Existence and uniqueness of a minimal solution}
In this section we provide sufficient conditions for the existence and uniqueness of a minimal solution $X_{\min}$ to \eqref{eq.TRiccati} and our result rely on the following fixed-point iteration
\begin{equation}\label{fixed-point}
 \begin{array}{rcl}
X_0&=&0,\\
 DX_{k+1}+X_{k+1}^TA &=& X_{k}^TBX_k-C,\quad k\geq 0.\\
 \end{array}
\end{equation}
\begin{theorem}\label{Th.1}
  The iterates computed by the fixed-point iteration \eqref{fixed-point} are such that
 $$X_{k+1}\geq X_k,\quad k\geq0,$$
and, if there exists a nonnegative matrix $Y$ such that $\mathcal{R}_T(Y)\geq0$, then $X_k\leq Y$ for any $k\geq 0$.
 Moreover, $\{X_k\}_{k\geq0}$ converges to the minimal nonnegative solution $X_{\min}$ to \eqref{eq.TRiccati}.
\end{theorem}
\begin{proof}
 We first show that $X_{k+1}\geq X_k$ for any $k\geq0$ by induction on $k$.
 For $k=0$, we have $X_1=\mathcal{S}_T^{-1}(-C)\geq0=X_0$ as $C\leq 0$.
 We now assume that $X_{\bar k}\geq X_{\bar k-1}$ for a certain $\bar k>0$ and we show that
 $X_{\bar k+1}\geq X_{\bar k}$. We have
 $$\begin{array}{rll}
 X_{\bar k+1}&=&\mathcal{S}_T^{-1}(X_{\bar k}^TBX_{\bar k}-C)=
 \mathcal{S}_T^{-1}(X_{\bar k}^TBX_{\bar k})+\mathcal{S}_T^{-1}(-C)=
 \mathcal{S}_T^{-1}(X_{\bar k}^TBX_{\bar k})+X_1+X_{\bar k}-X_{\bar k}\\
 &&\\
 &=&\mathcal{S}_T^{-1}(X_{\bar k}^TBX_{\bar k})+X_1+X_{\bar k}-\mathcal{S}_T^{-1}(X_{\bar k-1}^TBX_{\bar k-1}-C)
  =\mathcal{S}_T^{-1}(X_{\bar k}^TBX_{\bar k}-X_{\bar k-1}^TBX_{\bar k-1})+X_{\bar k}.
 \end{array}
$$
 Clearly, $X_{\bar k}^T\geq X_{\bar k-1}^T,$ as $X_{\bar k}\geq X_{\bar k-1}$ by inductive hypothesis.
 Therefore, recalling that $B\geq0$,  we have $$X_{\bar k}^TBX_{\bar k}-X_{\bar k-1}^TBX_{\bar k-1}\geq0$$ and
 $X_{\bar k+1}$ is thus nonnegative.

 We now suppose that there exists a nonnegative $Y$ such that $\mathcal{R}_T(Y)\geq0$ and we show that
 $X_k\leq Y$ for any $k\geq 0$ by induction on $k$ once again. The result is straightforward for $k=0$ as $X_0=0$.
We now assume that $X_{\bar k}\leq Y$ for a certain $\bar k>0$ and we show that
 $X_{\bar k+1}\leq Y$. Since $X_{\bar k}\leq Y$ and $B\geq0$, $X_{\bar k}^TBX_{\bar k}\leq Y^TBY$ so that
  $-X_{\bar k}^TBX_{\bar k}\geq -Y^TBY$. We can thus write
  $$  0\leq DY+Y^TA-Y^TBY+C\leq DY+Y^TA-X_{\bar k}^TBX_{\bar k}+C,$$
  and since $-X_{\bar k}^TBX_{\bar k}+C=-DX_{\bar k+1}-X_{\bar k+1}^TA$ by definition, we get
  $$  0\leq  DY+Y^TA-DX_{\bar k+1}-X_{\bar k+1}^TA.$$
  This means that $\mathcal{S}_T(Y-X_{\bar k+1})\geq0$ which implies $Y\geq X_{\bar k+1}$.

  In conclusion, $\{X_k\}_{k\geq0}$ is a nondecreasing, nonnegative sequence bounded from above and it thus has a finite
  limit $\lim_{k\rightarrow+\infty}X_k=X_{\min}\leq 0$. Taking the limit in both sides of \eqref{fixed-point} shows that
  $X_{\min}$ is a solution of the equation $\mathcal{R}_T(X)=0$. Moreover, $X_{\min}$ is the minimal nonnegative solution
  as we showed that $X_{\min}\leq Y$ for any nonnegative $Y$ such that $\mathcal{R}_T(Y)\geq0$.
 \end{proof}
A similar result has been shown in \cite[Theorem 2.3]{Guo2001} for the (standard) nonsymmetric Riccati equation.

%%%%%%%%%%%%%%%%%%%%%%%%%%%%%%%%%%%%%%%%%%%%%%%
\section{The (inexact) Newton-Kleinman method}\label{The (inexact) Newton-Kleinman method}
The fixed-point iteration \eqref{fixed-point} may be not very well-suited for the actual computation of the minimal solution $X_{\min}$ and a Newton-Kleinman-like method can be more effective for this task.

The $k$-th iteration of the Newton method is defined as
$$\mathcal{R}'_T[X](X_{k+1}-X_k)=-\mathcal{R}_T(X_k),$$
where $\mathcal{R}'_T[X]$ denotes the Fr\'{e}chet derivative of $\mathcal{R}_T$ at $X$. For the
nonsymmetric T-Riccati operator, we have
$$\mathcal{R}'_T[X](Y)=DY+Y^TA-Y^TBX-X^TBY=(D-X^TB)Y+Y^T(A-BX),$$
and therefore the $(k+1)$-st iterate of the Newton method is the solution of the T-Sylvester equation
\begin{equation}\label{Newton_step}
(D-X_k^TB)X_{k+1}+X_{k+1}^T(A-BX_{k})=-X_{k}^TBX_{k}-C.
\end{equation}
Depending on the problem size $n$, different
state-of-the-art methods can be employed for the solution of the equations \eqref{Newton_step}.
See, e.g., \cite{DeTeran2011,Dopico2016}.
However, we first need to guarantee that the sequence $\{X_k\}_{k\geq 0}$ generated by \eqref{Newton_step} is well-defined and it converges to $X_{\min}$; this is the topic of the next section.

%%%%%%%%%%%%%%%%%%%%%%%%%%%%%%%%%%%%%%%%%%%%%%%

\subsection{A convergence result}\label{A convergence result}

In this section we prove the convergence properties of the Newton-Kleinman method \eqref{Newton_step}. To this end, we first recall a couple of classic results about M-matrices.
See, e.g., \cite[Chapter 6]{Berman1994}.
\begin{lemma}\label{lemma_Mmatrix}
 Let $A$ be a Z-matrix. Then $A$ is a nonsingular M-matrix if and only if there exists a nonnegative vector $v$ such that $Av>0$. Moreover,
 if $A$ is a nonsingular M-matrix and $B\geq A$ is a Z-matrix, then $B$ is also a nonsingular M-matrix.
\end{lemma}

To prove the convergence of the Newton method to the minimal nonnegative solution $X_{\min}$ to \eqref{eq.TRiccati}, we also need the following lemma.
\begin{lemma}\label{lemma_Xmin}
 Assume there exists a matrix $\widebar Y$ such that $\mathcal{R}_T(\widebar Y)>0$, then $I\otimes (D-X_{\min}^TB)+((A-BX_{\min})^T\otimes I)\Pi$ is a nonsingular M-matrix.
\end{lemma}
\begin{proof}
Since Assumption~\ref{Ass1} holds,
$I\otimes D+(A^T\otimes I)\Pi=rI_{n^2}-N$, $N\geq0$, $r>\rho(N)$, and we can write
$$
\begin{array}{rll}
I\otimes (D-X_{\min}^TB)+((A-BX_{\min})^T\otimes I)\Pi&=&
I\otimes D+(A^T\otimes I)\Pi-(I\otimes X_{\min}^TB+
(BX_{\min})^T\otimes I)\Pi)\\
&&\\
&=&rI-\underbrace{(N+(I\otimes X_{\min}^TB+
(BX_{\min})^T\otimes I)\Pi)}_{\geq0},\\
\end{array}
$$
as $B,$ $X_{\min}\geq0$. Therefore, $I\otimes (D-X_{\min}^TB)+((A-BX_{\min})^T\otimes I)\Pi$ is a Z-matrix.

 Moreover,
 $$\begin{array}{rll}
    (D-X_{\min}^TB)(\widebar Y-X_{\min})+(\widebar Y-X_{\min})^T(A-BX_{\min})&=&D\widebar Y-X_{\min}^TB\widebar Y-DX_{\min}+X_{\min}^TBX_{\min}\\
    &&+\widebar Y^TA-\widebar Y^TBX_{\min}-X_{\min}^TA+X_{\min}^TBX_{\min}.
   \end{array}
$$
 Since $\mathcal{R}_T(X_{\min})=0$, $-DX_{\min}-X_{\min}^TA+X_{\min}^TBX_{\min}=C$. Moreover, adding and subtracting $\widebar Y^TB\widebar Y$ we get
 $$ (D-X_{\min}^TB)(\widebar Y-X_{\min})+(\widebar Y-X_{\min})^T(A-BX_{\min})=\mathcal{R}_T(\widebar Y)+
 (\widebar Y-X_{\min})^TB(\widebar Y-X_{\min}).
$$
 To conclude, we notice that $\widebar Y-X_{\min}\geq 0$ as $X_{\min}$ is the minimal solution to \eqref{eq.TRiccati} and $\mathcal{R}_T(\widebar Y)>0$.
 Therefore,
 $$ (D-X_{\min}^TB)(\widebar Y-X_{\min})+(\widebar Y-X_{\min})^T(A-BX_{\min})\geq\mathcal{R}_T(\widebar Y)>0.
$$
This means that $\mbox{vec}(\widebar Y-X_{\min})$ is a nonnegative vector such that
$(I\otimes (D-X_{\min}^TB)+((A-BX_{\min})^T\otimes I)\Pi)\mbox{vec}(\widebar Y-X_{\min})>0$ and
$I\otimes (D-X_{\min}^TB)+((A-BX_{\min})^T\otimes I)\Pi$ is thus a nonsingular M-matrix thanks to Lemma~\ref{lemma_Mmatrix}.
\end{proof}

\begin{theorem}
 If the assumptions of Lemma~\ref{lemma_Xmin} hold, the sequence $\{X_k\}_{k\geq0}$ computed by
 the Newton method \eqref{Newton_step} with $X_0=0$ is well-defined and $X_k\leq X_{k+1}\leq X_{\min}$ for any $k\geq0$.
 Moreover $\{X_k\}_{k\geq0}$ converges to the minimal nonnegative solution $X_{\min}$ to \eqref{eq.TRiccati}.
\end{theorem}
\begin{proof}
 For the Newton method \eqref{Newton_step} with $X_0=0$, the matrix $X_1$ is given by
 $$DX_1+X_1^TA=-C.$$
 Since the T-Sylvester operator $\mathcal{S}_T$ has a nonnegative inverse by Assumption~\ref{Ass1} and $-C\geq0$, the first
 iterate $X_1$ is nonnegative.
 Therefore the statements
 $$X_k\leq X_{k+1},\quad X_k\leq X_{\min}, \quad I\otimes (D-X_k^TB)+((A-BX_k)^T\otimes I)\Pi \mbox{ is an M-matrix},$$
 hold for $k=0$. We now assume that they hold for a certain $\bar k>0$ and we show them for $\bar k+1$. We start proving that
 $X_{\bar k+1}\geq X_{\bar k}$. By definition, we have
\begin{equation}\label{eq:TSylv1}
 (D-X_{\bar k}^TB)X_{\bar k+1}+X_{\bar k+1}^T(A-BX_{\bar k})=-X_{\bar k}^TBX_{\bar k}-C,
\end{equation}
 so that
 $$(D-X_{\bar k}^TB)(X_{\bar k+1}-X_{\bar k})+(X_{\bar k+1}-X_{\bar k})^T(A-BX_{\bar k})=-DX_{\bar k}-X_{\bar k}^TA
 +X_{\bar k}^TBX_{\bar k}-C. $$
We can write
{\small
 $$\begin{array}{rll}
 -DX_{\bar k}-X_{\bar k}^TA+X_{\bar k}^TBX_{\bar k}-C&=&-(D-X_{\bar k-1}^TB)X_{\bar k}-X_{\bar k}^T(A-BX_{\bar k-1})-X_{\bar k-1}^TBX_{\bar k}
 -X_{\bar k}^TBX_{\bar k-1}+X_{\bar k}^TBX_{\bar k}-C\\
 &&\\
 &=& X_{\bar k-1}^TBX_{\bar k-1}+C-X_{\bar k-1}^TBX_{\bar k} -X_{\bar k}^TBX_{\bar k-1}+X_{\bar k}^TBX_{\bar k}-C\\
 &&\\
 &=&(X_{\bar k}-X_{\bar k-1})^TB(X_{\bar k}-X_{\bar k-1})\geq0,\\
   \end{array}
 $$
 }
 since $X_{\bar k}\geq X_{\bar k-1}$ and $B\geq0$. If $\mathcal{S}^{(k)}_T(X):=(D-X_{\bar k}^TB)X+X^T(A-BX_{\bar k})$,
 then $(\mathcal{S}^{(k)}_T)^{-1}$ is nonnegative as the matrix $I\otimes (D-X_{\bar k}^TB)+((A-BX_{\bar k})^T\otimes I)\Pi$
 is a nonsingular M-matrix by inductive hypothesis. Therefore, $X_{\bar k+1}-X_{\bar k}\geq0$.

 We now show that $X_{k+1}\leq X_{\min}$. Considering again \eqref{eq:TSylv1}, we see that
% Once again,
% %%
% $$(D-X_{\bar k}^TB)X_{\bar k+1}+X_{\bar k+1}^T(A-BX_{\bar k})=-X_{\bar k}^TBX_{\bar k}-C,$$
% %%
% so that
 %%
 %{\small
 $$(D-X_{\bar k}^TB)(X_{\bar k+1}-X_{\min})+(X_{\bar k+1}-X_{\min})^T(A-BX_{\bar k})=-DX_{\min}-X_{\min}^TA+X_{\bar k}^TBX_{\min}+X_{\min}^TBX_{\bar k}
 -X_{\bar k}^TBX_{\bar k}-C. $$
 %}
%%
 We change sign and by adding and subtracting $X_{\min}^TBX_{\min}$ in the right-hand side, we get
 {\small
 $$\begin{array}{rll}
(D-X_{\bar k}^TB)(X_{\min}-X_{\bar k+1})+(X_{\min}-X_{\bar k+1})^T(A-BX_{\bar k})&=&DX_{\min}+X_{\min}^TA-X_{\bar k}^TBX_{\min}-X_{\min}^TBX_{\bar k}
 +X_{\bar k}^TBX_{\bar k}\\
 && +\, C+X_{\min}^TBX_{\min}-X_{\min}^TBX_{\min}\\
 &&\\
 &=&(X_{\min}-X_{\bar k})^TB(X_{\min}-X_{\bar k})\\
 &&\\
 &\geq&0,\\
   \end{array}
 $$
 }
where we have used the fact that $\mathcal{R}_T(X_{\min})=0$, $X_{\min}\geq X_{\bar k}$ and $B\geq0$. Since $\mathcal{S}^{(k)}_T$
has a nonnegative inverse we conclude that $X_{\min}-X_{\bar k+1}\geq0$.

The last statement we have to prove is that the matrix
$I\otimes (D-X_{\bar k+1}^TB)+((A-BX_{\bar k+1})^T\otimes I)\Pi$ is a nonsingular M-matrix. Since Assumption~\ref{Ass1} holds,
$I\otimes D+(A^T\otimes I)\Pi=rI_{n^2}-N$, $N\geq0$, $r>\rho(N)$, and we can write
$$
\begin{array}{rll}
I\otimes (D-X_{\bar k+1}^TB)+((A-BX_{\bar k+1})^T\otimes I)\Pi&=&
I\otimes D+(A^T\otimes I)\Pi-(I\otimes X_{\bar k+1}^TB+
(BX_{\bar k+1})^T\otimes I)\Pi)\\
&&\\
&=&rI-\underbrace{(N+(I\otimes X_{\bar k+1}^TB+
(BX_{\bar k+1})^T\otimes I)\Pi)}_{\geq0},\\
\end{array}
$$
as $B,$ $X_{\bar k+1}\geq0$. Therefore, $I\otimes (D-X_{\bar k+1}^TB)+((A-BX_{\bar k+1})^T\otimes I)\Pi$ is a Z-matrix.
Moreover, $$I\otimes (D-X_{\bar k+1}^TB)+((A-BX_{\bar k+1})^T\otimes I)\Pi\geq
I\otimes (D-X_{\min}^TB)+((A-BX_{\min})^T\otimes I)\Pi$$ since $X_{\bar k+1}\leq X_{\min}$ and
$I\otimes (D-X_{\min}^TB)+((A-BX_{\min})^T\otimes I)\Pi$ is a nonsingular M-matrix by Lemma~\ref{lemma_Xmin}. The matrix
$I\otimes (D-X_{\bar k+1}^TB)+((A-BX_{\bar k+1})^T\otimes I)$ is thus a nonsingular M-matrix by Lemma~\ref{lemma_Mmatrix}.

In conclusion, the Newton sequence $\{X_k\}_{k\geq0}$ is well-defined, nondecreasing and bounded from above.
Therefore, $\{X_k\}_{k\geq0}$ has a finite limit $X_*$ and, by taking the limit in both sides of \eqref{Newton_step},
it is easy to show that is also a solution of $\mathcal{R}_T(X)=0$. Moreover,
we can show $X_k\leq H$ for any $k\geq0$ and $H\geq0$ such that $\mathcal{R}_T(H)\geq0$ by induction.
Since the inequality is preserved for $k\rightarrow+\infty$, $X_*\leq H$ and
$X_*$ is thus the minimal solution of $\mathcal{R}_T(X)=0$, i.e., $X_*=X_{\min}$.
\end{proof}

If $n$ is moderate, say $n\leq \mathcal{O}(10^3)$, dense methods based on some decomposition of the coefficient matrices can be employed to solve the T-Sylvester equations in \eqref{Newton_step}. For instance, in \cite[Section 3]{DeTeran2011} an algorithm based on the generalized Schur decomposition of the pair $(D,A^T)$ is presented for efficiently solving a T-Sylvester equation of the form $DX+X^TA=C$.

If the problem dimension does not allow for dense matrix operations, equations \eqref{Newton_step} must be solved iteratively. The iterative solution of the T-Sylvester equations may introduce some inexactness in the Newton scheme leading to the so-called inexact Newton-Kleinman method and affecting the convergence features of the latter. By using tools similar to the ones presented in \cite{Benner2016},
in the next section we show how a specific line search guarantees the convergence of the inexact Newton method.

%%%%%%%%%%%%%%%%%%%%%%%%%%%%%%%%%%%%%%%%%%%%%%%

\subsection{The large-scale setting}\label{The large-scale setting}
In this section, we consider T-Riccati equations of large dimension. In this setting, unless the data $A$, $B$, $C$ and $D$ are equipped with some particular structure, equation \eqref{eq.TRiccati} is not numerically tractable. For instance,
the solution $X$ would be, in general, a dense $n\times n$ matrix that cannot be stored.
Therefore, as already mentioned, we assume that the matrices $A$ and $D$ are such that the matrix-vector products $Av$ and $Dw$ are easily computable in $\mathcal{O}(n)$ flops for any $v,w\in\mathbb{R}^n$. This is the case when, for instance, $A$ and $D$ are sparse. Moreover, we assume $B$ and $C$ to be low rank, namely $B=B_1B_2^T$, $B_1$, $B_2\in\mathbb{R}^{n\times p}$, and $C=C_1^TC_2$, $C_1$, $C_2\in\mathbb{R}^{q\times n}$, where $p+q\ll n$.
Equation \eqref{eq.TRiccati} can thus be written as
\begin{equation}\label{Eq.TRiccati_low-rank}
\mathcal{R}_T(X)=DX+X^TA-X^TB_1B_2^TX+C_1^TC_2=0,
\end{equation}
and low-rank approximations to $X$ are sought, namely we aim to compute and store only a couple of low-rank matrices $P_1,P_2\in\mathbb{R}^{n\times t}$, $t\ll n$, such that $P_1P_2^T\approx X$.

The results presented in the previous section are still valid also in the large-scale setting, for equation~\eqref{Eq.TRiccati_low-rank}.
The Newton method can be still applied and the $(k+1)$-st iterate can be computed by solving the equation
\begin{equation}\label{Newton_step_low-rank}
 (D-X_k^TB_1B_2^T)X_{k+1}+X_{k+1}^T(A-B_1B_2^TX_k)=-X_kB_1B_2^TX_k-C_1^TC_2.
\end{equation}
However, due to the large dimension of the problem, the exact solution to \eqref{Newton_step_low-rank} cannot be computed and only an approximation $\widetilde X_{k+1}\approx X_{k+1}$ can be constructed by, e.g., the projection methods presented in \cite{Dopico2016}.

The iterative solution of equations \eqref{Eq.TRiccati_low-rank} introduces some inexactness in the Newton scheme leading to an  inexact Newton method. The convergence result stated in Theorem~\ref{Th.1} no longer holds for the inexact variant of the Newton procedure and a line search has to be performed to ensure the convergence of the overall scheme.

Given a nonsymmetric $X_k\in\mathbb{R}^{n\times n}$, $\alpha>0$ and $\eta_k\in(0,1)$, we want to compute a matrix $S_k\in\mathbb{R}^{n\times n}$ such that
\begin{equation}\label{Eq.1_linesearch}
\|\mathcal{R}_T'[X_k](S_k)+\mathcal{R}_T(X_k)\|_F\leq \eta_k\|\mathcal{R}_T(X_k)\|_F,
\end{equation}
and then define the next iterate of the inexact Newton-Kleinman scheme as
\begin{equation}\label{iterate_linesearch}
X_{k+1}:=X_k+\lambda_k S_k,
\end{equation}
where the step size $\lambda_k>0$ is such that
\begin{equation}\label{decrease_condition}
 \|\mathcal{R}_T(X_k+\lambda_k S_k)\|_F\leq (1-\lambda_k\alpha)\|\mathcal{R}_T(X_k)\|_F,
\end{equation}
while $\lambda_k$ is not too small.

If we define the Newton step residual
\begin{equation}\label{def_L_k+1}
 \mathcal{R}_T'[X_k](S_k)+\mathcal{R}_T(X_k)=:L_{k+1},
\end{equation}
then equation \eqref{Eq.1_linesearch} can be written as $\|L_{k+1}\|_F\leq\eta_k\|\mathcal{R}_T(X_k)\|_F$. Moreover, writing explicitly the left-hand side in \eqref{def_L_k+1} we have
$$(D-X_k^TB_1B_2^T)(X_k+S_k)+(X_k+S_k)^T(A-B_1B_2^TX_k)+X_k^TB_1B_2^TX_k+C_1^TC_2=L_{k+1},
$$
so that the matrix $\widetilde X_{k+1}:=X_k+S_k$ is the solution of the T-Sylvester equation
\begin{equation}\label{inexactNewtonstep}
 (D-X_k^TB_1B_2^T)\widetilde X_{k+1}+\widetilde X_{k+1}^T(A-B_1B_2^TX_k)=-X_k^TB_1B_2^TX_k-C_1^TC_2+L_{k+1}.
\end{equation}
The matrix $L_{k+1}$ is never computed and the notation in \eqref{inexactNewtonstep} is used only to indicate that $\widetilde X_{k+1}$ is an inexact solution to the equation \eqref{Newton_step} such that the residual norm $\|L_{k+1}\|$ is sufficiently small. Once $\widetilde X_{k+1}$ is computed, we recover $S_k$ by $S_k=\widetilde X_{k+1}-X_k$
and the new iterate can be defined as in \eqref{iterate_linesearch}.

The T-Riccati residual
at $X_{k+1}$ can be written as
$$   \mathcal{R}_T(X_{k+1})=\mathcal{R}_T(X_k+\lambda_kS_k)=(1-\lambda_k)\mathcal{R}_k+\lambda_kL_{k+1}-\lambda_k^2S_k^TB_1B_2^TS_k,
$$
and, if $\eta_k\leq\widebar\eta <1$ and $\alpha\in(0,1-\widebar\eta)$, we have
$$\|\mathcal{R}_T(X_k+\lambda S_k)\|_F\leq(1-\lambda)\|\mathcal{R}_k\|_F+\lambda\|L_{k+1}\|_F+\lambda^2\|S_k^TB_1B_2^TS_k\|_F\leq (1-\alpha\lambda)\|\mathcal{R}_k\|_F,
$$
for all $\lambda\in(0,(1-\alpha-\widebar\eta )\frac{\|\mathcal{R}_T(X_k)\|_F}{\|S_k^TB_1B_2^TS_k\|_F}]$. In particular, the sufficient decrease condition \eqref{decrease_condition} is satisfied for all $\lambda$'s in the latter interval.

For the actual computation of the step size $\lambda_k$ we mimic the derivation given in \cite[Section 3]{Benner2016} for the algebraic Riccati equation, and we exploit the expression of the residual norm $\|\mathcal{R}_T(X_k+\lambda S_k)\|_F^2$ in terms of a quartic polynomial $p_k$ in $\lambda$. In particular,
\begin{equation}
p_k(\lambda)= \|\mathcal{R}_T(X_k+\lambda S_k)\|_F^2=(1-\lambda)^2\alpha_k+\lambda^2\beta_k+\lambda^4\delta_k+2\lambda(1-\lambda)\gamma_k-2\lambda(1-\lambda)\epsilon_k-2\lambda^3\xi_k,
\end{equation}
where
\begin{equation}\label{coefficients_linesearch}
\begin{array}{ll}
   \alpha_k=\|\mathcal{R}_T(X_k)\|_F^2, & \beta_k=\|L_{k+1}\|_F^2,\\
   \gamma_k=\langle\mathcal{R}_T(X_k),L_{k+1}\rangle_F, & \delta_k=\|S_k^TB_1B_2^TS_k\|_F^2,\\
   \epsilon_k=\langle\mathcal{R}_T(X_k),S_k^TB_1B_2^TS_k\rangle_F,& \xi_k=\langle L_{k+1},S_k^TB_1B_2^TS_k\rangle_F.
  \end{array}
\end{equation}

  The first derivative of $p_k(\lambda)$ is given by $$p_k'(\lambda)=-2(1-\lambda)\alpha_k+2\lambda\beta_k+4\lambda^3\delta_k+2(1-\lambda)\gamma_k-2\lambda(2+3\lambda)\epsilon_k-6\lambda^2\xi_k,$$ so that
  $$p_k'(0)=-2\alpha_k+2\gamma_k\leq (\eta_k-1)\|\mathcal{R}_T(X_k)\|_F^2<0$$ as $\eta_k\in(0,1)$ and $S_k$ is thus a descent direction.

The step size $\lambda_k$ can be computed by exploiting the expression of the T-Riccati residual norm in terms of $p_k(\lambda)$. If $\theta_k:=\min\{1,(1-\alpha-\widebar\eta )\sqrt{\alpha_k/\delta_k}\}$, we suggest to compute $\lambda_k$ as
\begin{equation}\label{stepsize_computation}
 \lambda_k:=\argmin_{(0,\theta_k]}p_k(\lambda).
\end{equation}
The choice of the interval $(0,\theta_k]$ is motivated by the fact that if $X_k$ and $\widetilde X_{k+1}$ are nonnegative matrices, then also $X_{k+1}=X_k+\lambda_k(\widetilde X_{k+1}-X_k)$ is nonnegative. Moreover, the sufficient decrease condition is satisfied for $\lambda_k\in(0,\theta_k]$.

Clearly \eqref{stepsize_computation} is not the only way to compute $\lambda_k$. For instance, in \cite[Section 3.2]{Benner2016} a step size computation based on the Armijo rule is explored in the case of the inexact Newton-Kleinman method applied to the algebraic Riccati equation and such approach can be adapted to our setting as well.

The inexact Newton-Kleinman method with line search is summarized in Algorithm~\ref{iNK_algorithm} and in the next theorem we show its convergence to the minimal solution $X_{\min}$.

%%%
\setcounter{AlgoLine}{0}
\begin{algorithm}
%\algsetup{linenosize=\small}
%\SetLine %% new algorithm2e: \SetAlgoLined
\caption{Inexact Newton-Kleinman method with line search $(X_0=0)$.\label{iNK_algorithm}}
\SetKwInOut{Input}{input}\SetKwInOut{Output}{output}
%%%%%%%%%%% INPUT %%%%%%%%%%%
\Input{$A,D\in\mathbb{R}^{n\times n},$ $B_1,B_2\in\RR^{n\times p}$, $C_1,C_2\in\RR^{q\times n}$,
 $\varepsilon>0$, $\widebar\eta\in(0,1)$, $\alpha\in(0,1-\widebar\eta).$}
%%%%%%%%%%% OUTPUT %%%%%%%%%%%
\Output{$X_k\in\mathbb{R}^{n\times n}$ approximate solution to \eqref{eq.TRiccati}.}
%%%%%%%%%%%%%%%%%%%%%%%%%%%%%%%%%%%
\BlankLine
\For{$k = 0,1,\dots,$ till convergence}{
 \If{$\|\mathcal{R}_T(X_k)\|_F<\varepsilon\cdot\|C_1^TC_2\|_F$}{
\nl \textbf{Stop} and return $X_k$}
\nl Select $\eta_k\in(0,\widebar \eta]$\\
\nl Compute $\widetilde X_{k+1}$ s.t.
$$(D-X_k^TB_1B_2^T)\widetilde X_{k+1}+\widetilde X_{k+1}^T(A-B_1B_2^TX_{k})^T=-X_{k}^TB_1B_2^TX_{k}-C_1^TC_2+L_{k+1}$$
where $\|L_{k+1}\|_F\leq\eta_k\|\mathcal{R}_T(X_k)\|_F$ \label{Alg.Lyap.solves}\\
\nl Set $S_k=\widetilde X_{k+1}-X_k$\\
\nl Compute $\lambda_k>0$ as in \eqref{stepsize_computation}\\
\nl Set $X_{k+1}=X_k+\lambda_kS_k$
}
\end{algorithm}
%%%
\newpage
\begin{theorem}\label{Theorem_linesearch}
 Let Assumption~\ref{Ass1} and Lemma~\ref{lemma_Xmin} hold and assume that for all $k\geq 0$, there exists a matrix $\widetilde X_{k+1}$ satisfying \eqref{inexactNewtonstep} where $\|L_{k+1}\|_F\leq\eta_k\|\mathcal{R}_T(X_k)\|_F$.

 \begin{itemize}
  \item[(i)] If the step sizes $\lambda_k$ are bounded away from zero, $\lambda_k\geq\lambda_{\min}>0$ for all $k$, then $\|\mathcal{R}_T(X_k)\|_F\rightarrow0$.
  \item[(ii)] If, in addition to (i), the matrices $L_{k+1}$ are nonnegative for all $k\geq 0$, then the sequence $\{X_k\}_{k\geq 0}$ generated by the inexact Newton-Kleinman method with $X_0=0$ is well-defined and $X_k\leq X_{k+1}\leq X_{\min}$. Moreover, $\{X_k\}_{k\geq 0}$ converges to the minimal solution $X_{\min}$ of \eqref{Eq.TRiccati_low-rank}.
 \end{itemize}

\end{theorem}
\begin{proof}
 The sufficient decrease condition \eqref{decrease_condition} implies that, for any $\ell\geq 0$,
 $$\begin{array}{rll}
 \|\mathcal{R}_T(X_0)\|_F&\geq&\displaystyle
 \|\mathcal{R}_T(X_0)\|_F-\|\mathcal{R}_T(X_{\ell+1})\|_F=
 \sum_{k=0}^\ell(\|\mathcal{R}_T(X_k)\|_F-\|\mathcal{R}_T(X_{k+1})\|_F)\\
 &\geq& \displaystyle  \sum_{k=0}^\ell\lambda_k\alpha\|\mathcal{R}_T(X_k)\|_F\geq0.
   \end{array}
 $$
 Taking the limit $\ell\rightarrow+\infty$ and using the fact that $\lambda_k\geq\lambda_{\min}>0$ for all $k$, we have $\|\mathcal{R}_T(X_k)\|_F\rightarrow 0$.

 The proof of $(ii)$ is given by induction on $k$. For $k=0$ we have
 $$D\widetilde X_1+\widetilde X_1^TA=-C_1^TC_2+L_1, \quad \|L_1\|_F\leq\eta_0\|C_1^TC_2\|_F.$$
 Since $I\otimes D+(A^T\otimes I)\Pi$ is a nonsingular M-matrix by assumption, $C_1^TC_2\leq 0$ and $L_1\geq 0$, the matrix $\widetilde X_1$ is nonnegative. Then $X_1:=\lambda_0\widetilde X_1\geq 0$ as $\lambda_0=\argmin_{(0,\theta_0]}p_0(\lambda)>0$. Moreover, $\mathcal{R}_T(X_0)=C_1^TC_2\leq 0$. Therefore, the statements
 $X_k\leq X_{k+1},$ $ X_k\leq X_{\min},$
 $\mathcal{R}_T(X_k)\leq 0,$ and
 $I\otimes (D-X_k^TB_1B_2^T)+((A-B_1B_2^TX_k)^T\otimes I)\Pi$ being a nonsingular M-matrix hold for $k=0$.

 We now assume they hold also for a certain $\bar k>0$ and we show them for $\bar k+1$. We have
 $$(D-X_{\bar k}^TB_1B_2^T)\widetilde X_{\bar k+1}+\widetilde X_{\bar k+1}^T(A-B_1B_2^TX_{\bar k})=-X_{\bar k}^TB_1B_2^TX_{\bar k}-C_1^TC_2+L_{\bar k +1}, \quad \|L_{\bar k +1}\|_F\leq\eta_{\bar k}\|\mathcal{R}_T(X_{\bar k})\|_F,$$
 so that
 $$(D-X_{\bar k}^TB_1B_2^T)(\widetilde X_{\bar k+1}-X_{\bar k})+(\widetilde X_{\bar k+1}-X_{\bar k})^T(A-B_1B_2^TX_{\bar k})=-DX_{\bar k}-X_{\bar k}^TA+X_{\bar k}^TB_1B_2^TX_{\bar k}-C_1^TC_2+L_{\bar k +1}.$$
 The right-hand side in the above expression can be written as
 {\small
 $$
 \begin{array}{rll}
    -DX_{\bar k}-X_{\bar k}^TA+X_{\bar k}^TB_1B_2^TX_{\bar k}-C_1^TC_2+L_{\bar k +1}&=& -(D-X_{\bar k -1}^TB_1B_2^T)X_{\bar k}-X_{\bar k}^T(A-B_1B_2^TX_{\bar k-1}) -X_{\bar k-1}^TB_1B_2^TX_{\bar k}\\
    &&-X_{\bar k}^TB_1B_2^TX_{\bar k-1}+X_{\bar k}^TB_1B_2^TX_{\bar k}-C_1^TC_2+L_{\bar k +1}.\\
   \end{array}
$$
}
 Recalling that $X_{\bar k}=(1-\lambda_{\bar k-1})X_{\bar k-1}+\lambda_{\bar k-1}\widetilde X_{\bar k}$ and that $\widetilde X_{\bar k}$ satisfies an equation of the form \eqref{Newton_step_low-rank}, a direct computation shows that
 \begin{align*}
 -DX_{\bar k}-X_{\bar k}^TA+X_{\bar k}^TB_1B_2^TX_{\bar k}-C_1^TC_2+L_{\bar k +1}=&\,
 (X_{\bar k}-X_{\bar k-1})^TB_1B_2^T(X_{\bar k}-X_{\bar k-1})-(1-\lambda_{\bar k-1})\mathcal{R}_T(X_{\bar k-1})\\
 & +L_{\bar k +1}+\lambda_{\bar k-1}L_{\bar k}\geq 0,\\
 \end{align*}
 as $X_{\bar k}\geq X_{\bar k-1}$ and $\mathcal{R}_T(X_{\bar k-1})\leq 0$ by inductive hypothesis, $B_1B_2^T\geq 0$, ,
 $L_{\bar k +1},L_{\bar k}\geq 0$, and $\lambda_{\bar k-1}\in(0,1]$.
 This implies that $\widetilde X_{k+1}\geq X_{\bar k}$ as the matrix $I\otimes (D-X_{\bar k}^TB_1B_2^T)+((A-B_1B_2^TX_{\bar k})^T\otimes I)\Pi$ is a nonsingular M-matrix by inductive hypothesis.

 Once $\lambda_{\bar k}$ is computed as in \eqref{stepsize_computation}, a direct computation shows that
$X_{\bar k+1}=(1-\lambda_{\bar k})X_{\bar k}+\lambda_{\bar k}\widetilde X_{\bar k+1}\geq X_{\bar k}$.

We now show that $X_{\bar k+1}\leq X_{\min}$. To this end we can show that $\widetilde X_{\bar k+1}\leq X_{\min}$ since $X_{\bar k+1}\leq \widetilde X_{\bar k+1}$. Indeed,
$$X_{\bar k+1}=(1-\lambda_{\bar k})X_{\bar k}+\lambda_{\bar k}\widetilde X_{\bar k+1}
\leq (1-\lambda_{\bar k})\widetilde X_{\bar k+1}+\lambda_{\bar k}\widetilde X_{\bar k+1}=\widetilde X_{\bar k+1}.$$
We have
{\small
\begin{align*}%{rll}
(D-X_{\bar k}^TB_1B_2^T)(\widetilde X_{\bar k+1}-X_{\min})+(\widetilde X_{\bar k+1}-X_{\min})^T(A-B_1B_2^TX_{\bar k})=&-DX_{\min}-X_{\min}^TA+X_{\min}^TB_1B_2^TX_{\bar k}\\
&+X_{\bar k}^TB_1B_2^TX_{\min}
-X_{\bar k}^TB_1B_2^TX_{\bar k}
 -C_1^TC_2+L_{\bar k +1},
\end{align*}
%$$
}
and by changing the sign, adding and subtracting $X_{\min}^TB_1B_2^TX_{\min}$ in the right-hand side, we get
$$
(D-X_{\bar k}^TB_1B_2^T)(\widetilde X_{\bar k+1}-X_{\min})+(\widetilde X_{\bar k+1}-X_{\min})^T(A-B_1B_2^TX_{\bar k})=(X_{\min}-X_{\bar k})^TB_1B_2^T(X_{\min}-X_{\bar k})+L_{k+1},
$$
where we used the fact that $\mathcal{R}_T(X_{\min})=0$. Since $X_{\min}\geq X_{\bar k}$ by inductive hypothesis, $B_1B_2^T$, $L_{k+1}\geq0$, the right-hand side in the above equation is nonnegative so that $\widetilde X_{\bar k+1}\leq X_{\min}$ thanks to the fact that $I\otimes (D-X_{\bar k}^TB_1B_2^T)+((A-B_1B_2^TX_{\bar k})^T\otimes I)\Pi$ is a nonsingular M-matrix.

To show that $I\otimes (D-X_{\bar k+1}^TB_1B_2^T)+((A-B_1B_2^TX_{\bar k+1})^T\otimes I)\Pi$ is a nonsingular M-matrix, we can use the same argument as in the proof of Theorem~\ref{Th.1} as $ X_{\bar k+1}\leq X_{\min}$.

The last statement we have to show is $\mathcal{R}_T(X_{\bar k+1})\leq 0$. We can write
\begin{align*}
\mathcal{R}_T(X_{\bar k+1})=&\,(D-X_{\bar k+1}^TB_1B_2^T)(X_{\bar k+1}-X_{\min})+(X_{\bar k+1}-X_{\min})^T(A-B_1B_2^TX_{\bar k+1})
 -DX_{\min}-X_{\min}^TA\\
 &+X_{\min}^TB_1B_2^TX_{\bar k+1}+X_{\bar k+1}^TB_1B_2^TX_{\min}
-X_{\bar k+1}^TB_1B_2^TX_{\bar k+1}
 -C_1^TC_2.
\end{align*}
Since $X_{\bar k+1}-X_{\min}\leq 0$ and $I\otimes (D-X_{\bar k+1}^TB_1B_2^T)+((A-B_1B_2^TX_{\bar k+1})^T\otimes I)\Pi$ is a nonsingular M-matrix, $(D-X_{\bar k+1}^TB_1B_2^T)(X_{\bar k+1}-X_{\min})+(X_{\bar k+1}-X_{\min})^T(A-B_1B_2^TX_{\bar k+1})\leq 0$ and we have
\begin{align*}
\mathcal{R}_T(X_{\bar k+1})\leq&
 -DX_{\min}-X_{\min}^TA+X_{\min}^TB_1B_2^TX_{\bar k+1}+X_{\bar k+1}^TB_1B_2^TX_{\min}
-X_{\bar k+1}^TB_1B_2^TX_{\bar k+1}
 -C_1^TC_2\\
 =& - (X_{\min}-X_{\bar k+1})^TB_1B_2^T(X_{\min}-X_{\bar k+1})\leq 0,
\end{align*}
as $X_{\min}\geq X_{\bar k+1}$ and $B_1B_2^T\geq 0$.

In conclusion, the sequence $\{X_k\}_{k\geq 0}$ computed by the inexact Newton-Kleinman method with $X_0=0$ and equipped with the line search \eqref{stepsize_computation} is well-defined, nondecreasing and bounded from above.
Therefore, $\{X_k\}_{k\geq 0}$ has a finite limit $X_*$ that is also a solution of the T-Riccati equation since
$$0=\lim_{k\rightarrow+\infty}\|\mathcal{R}_T(X_k)\|_F=\|\mathcal{R}_T(\lim_{k\rightarrow+\infty}X_k)\|_F=\|\mathcal{R}_T(X_*)\|_F.
$$
Moreover, it is easy to show that $X_*\leq H$ for every nonnegative $H$ such that $\mathcal{R}_T(H)\geq 0$, hence $X_*=X_{\min}$.
\end{proof}

The assumption on the nonnegativity of $L_{k+1}$ may remind the reader of the hypothesis made in \cite{Feitzinger2009} for proving the converge of the inexact Newton-Kleinman method applied to the standard algebraic Riccati equation. Indeed, in \cite[Theorem 4.4]{Feitzinger2009}, the matrix $L_{k+1}$ is supposed to be positive semidefinite for all $k$. However, as outlined in \cite{Benner2016}, this condition is hard to meet in practice and in \cite[Theorem 10]{Benner2016} a different approach is used for showing the convergence of the inexact Newton scheme. In our setting we do not see any particular drawback in assuming
$L_{k+1}$ nonnegative for every $k$. Moreover, if the projection method presented in \cite{Dopico2016} is employed for the computation of $\widetilde X_{k+1}$, then the nonnegativity of $L_{k+1}$ may be further explored by exploiting the explicit form of this residual matrix given in \cite[Proposition 4.3]{Dopico2016}. However, this is beyond the scope of this paper.

The line search \eqref{iterate_linesearch} can be performed also when the exact solution to \eqref{inexactNewtonstep} can be computed
as in the case of moderate $n$.
If $\|L_{k+1}\|_F=0$ for all $k$ in \eqref{inexactNewtonstep}, it is easy to show that the quartic polynomial $p_k(\lambda)$ has a local minimizer in $(0,2]$ for all $k$ and we can replace the computation of the step size \eqref{stepsize_computation} by $\lambda_k:=\argmin_{(0,2]}p_k(\lambda)$; Theorem~\ref{Theorem_linesearch} still holds. This procedure may improve the convergence rate of the exact Newton-Kleinman method, especially for the first iterations, as shown in \cite{Benner1998} for the standard algebraic Riccati equation. See Example~\ref{Ex.1} in section~\ref{Numerical examples}.

%%%%%%%%%%%%%%%%%%%%%%%%%%%%%%%%%%%%%%%%%%%%%%%
\subsection{Implementation details}\label{Implementation details}
In this section, we present some details for an efficient implementation of Algorithm~\ref{iNK_algorithm}.

First of all, we recall that the computation of the Frobenius norm of low-rank matrices does not need to assemble any $n\times n$ dense matrix. For instance, only $q\times q$ matrices are actually involved in the computation of $\|C_1^TC_2\|_F$ as
$$\|C_1^TC_2\|_F^2=\text{trace}(C_2^TC_1C_1^TC_2)=\text{trace}((C_1C_1^T)(C_2C_2^T)).$$

The most expensive part of Algorithm~\ref{iNK_algorithm} is the solution of the large-scale T-Sylvester equations in line~3.
These equations can be solved, e.g., by employing the projection method presented in \cite{Dopico2016}. Given the T-Sylvester equation
$$DX+X^TA=-C_1^TC_2,$$
an approximate solution $X_m\in\mathbb{R}^{n\times n}$ of the form $X_m=V_mY_mW_m^T\approx X$ is constructed, where the orthonormal columns
of $V_m,W_m\in\mathbb{R}^{n\times \ell}$ span suitable subspaces $\mathcal{K}_{V_m}$ and $\mathcal{K}_{W_m}$ respectively, i.e.,
$\mathcal{K}_{V_m}=\text{Range}(V_m)$ and $\mathcal{K}_{W_m}=\text{Range}(W_m)$. We will always assume that $V_m$ and $W_m$ have full rank so that $\text{dim}(\mathcal{K}_{V_m})=\text{dim}(\mathcal{K}_{W_m})=\ell$. If this is not the case, deflation strategies as the ones presented in \cite{Gutknecht2006} can be implemented to overcome the possible linear dependence of the spanning vectors.
The $\ell\times\ell$ matrix $Y_m$ is computed by imposing a Petrov-Galerkin condition on the residual matrix $R_m=DX_m+X_m^TA+C_1^TC_2$ with respect to the space $\mathcal{K}_{W_m}\otimes \mathcal{K}_{W_m}$. This condition is equivalent to computing $Y_m$ by solving the reduced T-Sylvester equation
\begin{equation}\label{projected_eq}
(W_m^TDV_m)Y_m+Y_m^T(V_m^TAW_m)=-(W_m^TC_1^T)(C_2W_m).
\end{equation}
See \cite[Section 3]{Dopico2016}.
Equation~\eqref{projected_eq} can be solved by employing, e.g., Algorithm 3.1 presented in \cite{DeTeran2011}
as the small dimension of the coefficient matrices allows for the computation of the generalized Schur decomposition of the pair
$(W_m^TDV_m,(V_m^TAW_m)^T)$.

The effectiveness of the projection framework presented in \cite{Dopico2016}
is strictly related to the choice of the approximation spaces $\mathcal{K}_{V_m}$ and
$\mathcal{K}_{W_m}$. In \cite{Dopico2016} it is shown how the selection of these spaces may depend on the location of the spectrum $\Lambda(A^{-T}D)$ of $A^{-T}D$.
In particular, if $\Lambda(A^{-T}D)$ is strictly contained in the unit disk,
it is suggested to select
$$\mathcal{K}_{V_m}=\mathbf{K}_m^\square(A^{-T}D,A^{-T}[C_1^T,C_2^T]), \quad\text{and}\quad \mathcal{K}_{W_m}=A^T\cdot\mathcal{K}_{V_m}=\mathbf{K}_m^\square(DA^{-T},[C_1^T,C_2^T]),$$
where
$$\mathbf{K}_m^\square(A^{-T}D,A^{-T}[C_1^T,C_2^T])=\text{Range}(\left[A^{-T}[C_1^T,C_2^T],A^{-T}DA^{-T}[C_1^T,C_2^T],\ldots,
(A^{-T}D)^{m-1}A^{-T}[C_1^T,C_2^T]\right]),$$
is the block Krylov subspace generated by $A^{-T}D$ and $A^{-T}[C_1^T,C_2^T]$. If instead $\Lambda(A^{-T}D)$ is well outside the unit disk, then the roles of $A$ and $D$ are reversed and we can choose
$$\mathcal{K}_{V_m}=\mathbf{K}_m^\square(D^{-1}A^T,D^{-1}[C_1^T,C_2^T]), \quad\text{and}\quad \mathcal{K}_{W_m}=D\cdot\mathcal{K}_{V_m}=\mathbf{K}_m^\square(A^{T}D^{-1},[C_1^T,C_2^T]).$$

However, in general, the spectrum of $A^{-T}D$ is neither strictly contained in the unit disk nor well outside it and the employment of the extended Krylov subspaces
\begin{equation}\label{selected_spaces}
\mathcal{K}_{V_m}=\mathbf{EK}_m^\square(A^{-T}D,A^{-T}[C_1^T,C_2^T]),\quad
\text{and}\quad
\mathcal{K}_{W_m}=A^T\cdot\mathbf{EK}_m^\square(A^{-T}D,A^{-T}[C_1^T,C_2^T]),
\end{equation}
where $\mathbf{EK}_m^\square(A^{-T}D,A^{-T}[C_1^T,C_2^T]):=\mathbf{K}_m^\square(A^{-T}D,A^{-T}[C_1^T,C_2^T])+\mathbf{K}_{m}^\square(D^{-1}A^T,D^{-1}[C_1^T,C_2^T])$,
is recommended in this case. It has been shown how the projection method based on the extended Krylov subspaces \eqref{selected_spaces} performs quite well  in most of the results reported in \cite[Section 7]{Dopico2016} and if this procedure fails to converge, then also the projection schemes based on the block Krylov subspaces above fail as well.
Therefore, we also adopt the extended Krylov subspaces \eqref{selected_spaces} as approximation spaces in the solution of the sequence of T-Sylvester equations \eqref{Newton_step_low-rank} arising from the inexact Newton-Kleinman scheme.

The coefficient matrix defining the equations in \eqref{Newton_step_low-rank} are of the form $D-X_k^TB_1B_2^T$ and $A-B_1B_2^TX_k$ so that the spaces
$$\mathbf{EK}_m^\square((A-B_1B_2^TX_k)^{-T}(D-X_k^TB_1B_2^T),(A-B_1B_2^TX_k)^{-T}[C_1^T,C_2^T,X_k^TB_1,X_k^TB_2]),$$
and
$$
(A-B_1B_2^TX_k)^T\cdot\mathbf{EK}_m^\square((A-B_1B_2^TX_k)^{-T}(D-X_k^TB_1B_2^T),(A-B_1B_2^TX_k)^{-T}[C_1^T,C_2^T,X_k^TB_1,X_k^TB_2]),
$$
have to be computed at each Newton step $k\geq 0$. Such constructions require to solve linear systems of the form $(A+MN^T)z=y$ where $M,N\in\mathbb{R}^{n\times p}$ are low-rank and the Sherman-Morrison-Woodbury (SMW) formula
$$(A-MN^T)^{-1}=A^{-1}+A^{-1}M(I-N^TA^{-1}M)^{-1}N^TA^{-1},$$
can be employed to this end. See, e.g.,
\cite[Equation (2.1.4)]{Golub2013}.

Algorithm~\ref{T_EKSM} summarizes the projection method for the solution of the $(k+1)$-st T-Sylvester equation~\eqref{Newton_step_low-rank} where we suppose that the $k$-th iterate $X_k$ is given in low-rank format, namely $X_k=P_{1,k}P_{2,k}^T$, $P_k\in\mathbb{R}^{n\times t_k}$, $t_k\ll n$.

%%%
\setcounter{AlgoLine}{0}
\begin{algorithm}
%\algsetup{linenosize=\small}
%\SetLine %% new algorithm2e: \SetAlgoLined
\caption{Extended Krylov subspace method for T-Sylvester equations.\label{T_EKSM}}
\SetKwInOut{Input}{input}\SetKwInOut{Output}{output}
%%%%%%%%%%% INPUT %%%%%%%%%%%
\Input{$A,D\in\mathbb{R}^{n\times n},$ $B_1,B_2\in\RR^{n\times p}$, $C_1,C_2\in\RR^{q\times n}$, $P_{1,k},P_{2,k}\in\mathbb{R}^{n\times t_k}$, $t_k\ll n$,
 $\varepsilon>0$, $m_{\max}>0$}
%%%%%%%%%%% OUTPUT %%%%%%%%%%%
\Output{$\widetilde P_{1,k+1},\widetilde P_{2,k+1}\in\mathbb{R}^{n\times t_{k+1}}$, $t_{k+1}\ll n$, s.t. $\widetilde X_{k+1}=\widetilde P_{1,k+1}\widetilde P_{2,k+1}^T$ is an approximate solution to \eqref{Newton_step_low-rank}}
%%%%%%%%%%%%%%%%%%%%%%%%%%%%%%%%%%%
\BlankLine
\nl Set ${\pmb \alpha}=P_{1,k}^TB_1$ and ${\pmb \beta}=P_{1,k}^TB_2$\\
\nl Set $H=[C_1^T,C_2^T,P_{2,k}{\pmb \alpha},P_{2,k}{\pmb \beta}]$\\
   \nl  Perform economy-size QR,
   $[H,(A-B_1(B_2^TP_k)P_k^T)^{-1}H]=[\mathcal{V}_1^{(1)},\mathcal{V}_1^{(2)}]
  {\pmb \gamma},$
  where ${\pmb \gamma},\in\RR^{4(q+p)\times 4(q+p)}$\\                                                                                          \nl Set $V_1= [\mathcal{V}_1^{(1)},\mathcal{V}_1^{(2)}]$ \\ %, $V_1^TB=\pmb{\gamma}$  \\
  \nl $W_1$ $\leftarrow$ orthonormalize the columns of $(A-B_1{\pmb \beta}^TP_{2,k}^T)^{T}V_1$\\
  \For{$m=1, 2,\dots,$ till $m_{\max}$}{%
\nl Compute next basis block $\mathcal{V}_{m+1}$ as in \cite{Dopico2016} and set $V_{m+1}=[V_{{m}},\mathcal{V}_{m+1}]$ \\
\nl $\mathcal{W}_{m+1}$ $\leftarrow$ orthonormalize the columns of $(A-B_1{\pmb \beta}^TP_{1,k}^T)^{T}\mathcal{V}_{m+1}$ w.r.t. $W_{m}$ \\
\nl Set $W_{m+1}=[W_m,\mathcal{W}_{m+1}]$\\
\nl Update $T_{m}=W_{m}^T(D-P_{2,k}{\pmb \alpha}B_2^T)V_{m}$, $K_{m}=V_{m}^T(A-B_1{\pmb \beta}^TP_{2,k}^T)W_{m}$  as in \cite{Dopico2016}  \\
\nl Update $G_1=W_m^T[C_1^T,P_{2,k}{\pmb \alpha}]$ and
  $G_2=W_m^T[C_2^T,P_{2,k}{\pmb \beta}^T]$\\
\nl Solve $T_mY_m+Y_m^TK_m=-G_1G_2^T$ \\
  \If{$\|L_{k+1}\|_F=\|(D-P_{2,k}{\pmb \alpha}B_2^T)(V_mY_mW_m^T)+ (V_mY_mW_m^T)^T(A-B_1{\pmb \beta}^TP_{2,k})+P_{2,k}{\pmb \alpha}{\pmb \beta}^TP_{2,k}+C_1^TC_2\|_F\leq\varepsilon$}{%
\nl \textbf{Break} and go to \textbf{13}}%
}%
\nl Factorize $Y_{m}$ and retain $\widehat Y_{1,m},\widehat Y_{2,m}\in\mathbb{R}^{4m(q+p)\times t_{k+1}}$, $t_{k+1}\leq 4m(q+p)$, $\widehat Y_{1,m}\widehat Y_{2,m}^T\approx Y_m$\\
\nl Set $\widetilde P_{1,k+1}=V_{m}\widehat Y_{1,m}$, $\widetilde P_{2,k+1}=W_{m}\widehat Y_{2,m}$\\
\end{algorithm}
%%%

To compute the residual norm
$\|L_{k+1}\|_F$ we do not need to construct the dense $n\times n$ residual
matrix $L_{k+1}=(D-P_{2,k}{\pmb \alpha}B_2^T)(V_mY_mW_m^T)+ (V_mY_mW_m^T)^T(A-B_1{\pmb \beta}^TP_{2,k})+P_{2,k}{\pmb \alpha}{\pmb \beta}^TP_{2,k}+C_1^TC_2$.
Indeed, it is easy to show that
$$\|L_{k+1}\|_F=\|{\pmb \tau}_{m+1,m}(e_m^T\otimes I_{4(p+q)})Y_m\|_F,$$
where ${\pmb \tau}_{m+1,m}:=W_{m+1}^T(D-P_{2,k}{\pmb \alpha}B_2^T) V_m$ and $e_m\in\RR^m$ is the $m$-th canonical basis vector of $\RR^m$.
See \cite[Proposition 5.1]{Dopico2016}. At the $k$-th iteration of the Newton-Kleinman scheme we can set $\varepsilon=\eta_k\|\mathcal{R}_T(X_k)\|_F$ as inner tolerance for Algorithm~\ref{T_EKSM}.

The computation of the coefficients in \eqref{coefficients_linesearch} needed for calculating the step-size $\lambda_k$ can be carried out at low cost. Indeed, even if it is not evident, all the quantities in
\eqref{coefficients_linesearch} consist of
inner products with low-rank matrices and they are thus cheap to evaluate as recalled at the beginning of this section. In
particular, if $X_k=P_{1,k}P_{2,k}^T$ is the $k$-th iterate of the Newton-Kleinman scheme and $\widetilde X_{k+1}=\widetilde P_{1,k+1}\widetilde P_{2,k+1}^T$ is the matrix computed by Algorithm~\ref{T_EKSM}, then we can write
\begin{align*}
 \mathcal{R}_T(X_k)&=DX_k+X_k^TA-X_k^TB_1B_2^TX_k+C_1C_2^T\\
   &= [DP_{1,k}, P_{2,k}, -P_{2,k}(P_{1,k}^TB_1), C_1^T][P_{2,k}, A^TP_{1,k}, P_{2,k}(P_{1,k}^TB_2), C_2^T]^T,\\
   \\
 L_{k+1}&=(D-X_k^TB_1B_2^T)\widetilde X_{k+1}+ \widetilde X_{k+1}^T(A-B_1B_2^TX_{k})+X_{k}^TB_1B_2^TX_{k}+C_1^TC_2  \\
 &=[D\widetilde P_{1,k+1}, -P_{2,k}{\pmb \alpha}\widetilde{\pmb \beta}^T,\widetilde P_{2,k+1}, \widetilde P_{2,k+1},
 \widetilde P_{2,k+1}\widetilde{\pmb \alpha}^T, C_1^T]\cdot \\
 & \quad \;[\widetilde P_{2,k+1},\widetilde P_{2,k+1}, A^T\widetilde P_{1,k+1}, -P_{2,k}({\pmb \beta}\widetilde{\pmb \alpha}^T), \widetilde P_{2,k+1}\widetilde {\pmb \beta}^T, C_2^T]^T,\\
\\
S_k&=\widetilde X_{k+1}-X_k=[\widetilde P_{1,k+1},- P_{1,k}][\widetilde P_{2,k+1}, P_{2,k}]^T,
\end{align*}
where ${\pmb \alpha}=P_{1,k}^TB_1$, ${\pmb \beta}=P_{1,k}^TB_2$,
$\widetilde{\pmb \alpha}=\widetilde P_{1,k+1}^TB_1$ and $\widetilde{\pmb \beta}=\widetilde P_{1,k+1}^TB_2$.

%%%%%%%%%%%%%%%%%%%%%%%%%%%%%%%%%%%%%%%%%%%%%%%
\section{Numerical examples} \label{Numerical examples}
In this section we report some results regarding the numerical solution of the nonsymmetric T-Riccati equation \eqref{eq.TRiccati}. Different instances of \eqref{eq.TRiccati} are considered and both the small-scale and the large-scale scenario are addressed.

When $n$ is moderate, the T-Sylvester equations arising from the Newton-Kleinman scheme \eqref{Newton_step} are solved by means of Algorithm~3.1 presented in \cite{DeTeran2011}. We show that also when equations \eqref{Newton_step} are solved exactly, a line search can improve the convergence rate of the Newton-Kleinman scheme by maintaining a monotone decrease in the residual norm. We always set the threshold for the relative residual norm to be equal to $10^{-12}$ for small $n$. Moreover, we report the number of iterations, i.e., the number of T-Sylvester equations solved, to achieve such accuracy, the final relative residual norm and the overall computational time in seconds.

For large problem dimensions, the inexact Newton-Kleinman method is employed in the solution of \eqref{eq.TRiccati} together with Algorithm~\ref{T_EKSM}
as inner solver. The tolerance for the outer relative residual norm achieved by the Newton scheme is set to $10^{-6}$ while the one for the inner solver changes as the iterations proceed accordingly to the discussion in section~\ref{Implementation details} where $\eta_k=1/(1+k^3)$.
Also in the large-scale setting we report the total number of T-Sylvester equations that need to be solved to get the desired accuracy, along with the average number of inner iterations, the final relative residual norm and the computational time for solving the problem. Moreover, since the memory requirements are one of the main issue in the numerical solution of large-scale matrix equations, we also document the storage demand of the solution process which corresponds to the dimension of the largest spaces \eqref{selected_spaces} constructed. The rank of the final numerical solution is reported to show that, at least in the tested examples, a low-rank approximate solution to \eqref{eq.TRiccati} can be sought.

All results were obtained with MATLAB R2017b \cite{MATLAB} on a Dell machine with two 2GHz
processors and 128 GB of RAM.

\begin{num_example}\label{Ex.1}
{\rm
 In the first example, we consider the same coefficient matrices as in \cite[Numerical test 7.1]{Dopico2016}. In particular, the matrices $D,A\in\mathbb{R}^{n\times n}$ come from the finite difference discretization on the unit square of the 2-dimensional differential operators
 $$\mathcal{L}_D(u)=-u_{xx}-u_{yy}+y(1-x)u_x+\gamma u, \quad \text{and}
 \quad \mathcal{L}_A(u)=-u_{xx}-u_{yy},
 $$
 respectively, and $\gamma=10^4$.
 Homogeneous Dirichlet boundary conditions are considered.

 We first tackle the case of moderate problem dimensions and choose $B,C\in\mathbb{R}^{n\times n}$ to be full random matrices.

 In Table~\ref{tab1} we report the results for different $n$.
 \begin{table}[!ht]
 \centering
 \caption{Example \ref{Ex.1}. Results for different values of (moderate) $n$. \label{tab1}}
\begin{tabular}{|r| r r r r|}
 \hline
 & $n$  & Its & Rel. Res & Time (secs) \\
 \hline
 w/o line search & \multirow{ 2}{*}{324} & 8 & 8.51e-15& 11.28 \\
 w/ line search & & 5 & 2.99e-14 & 7.54 \\
\hline
 w/o line search & \multirow{ 2}{*}{784} & 10 & 8.62e-14& 99.94 \\
 w/ line search & & 8 & 2.32e-14 & 73.73 \\
\hline

\end{tabular}
 \end{table}

 For this example, the exact line search discussed at the end of section~\ref{The large-scale setting} is effective in decreasing the number of iterations necessary to achieve the prescribed accuracy leading to a speed-up of the solution process. In particular, a small step-size $\lambda_1$ is computed at the first iteration avoiding an increment in the relative residual norm and allowing us to faster reach the region where quadratic convergence occurs. This is apparent from Figure~\ref{Ex.1_Fig.1} where the relative residual norms produced by the Newton-Kleinman method with and without line search are plotted for the case $n=784$.
 We can appreciate how a monotone decrease in the relative residual norm is obtained if the line search is performed.

 \begin{figure}
  \centering
  \caption{Example~\ref{Ex.1}. Relative residual norms produced by the Newton-Kleinman with and without line search for $n=784$.}
  \label{Ex.1_Fig.1}
  	\begin{tikzpicture}
    \begin{semilogyaxis}[width=0.8\linewidth, height=.27\textheight,
      legend pos = south west,
      xlabel = $k$, ylabel = $\|\mathcal{R}_T(X_k)\|_F/\|C_1^TC_2\|_F$,xmin=0, xmax=10, ymax = 1e4]
      \addplot+[thick] table[x index=0, y index=1]  {data1.dat};
       \addplot+ [thick]table[x index=0, y index=2]  {data1.dat};
        \legend{w/o line search,w/ line search};
    \end{semilogyaxis}
  \end{tikzpicture}

\end{figure}
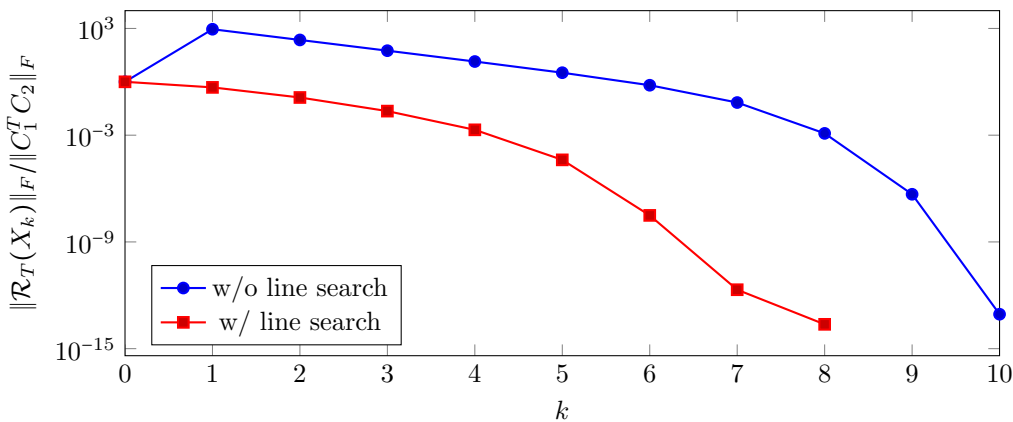

 In the large-scale setting, we consider low-rank matrices $B=B_1B_2^T$, $B_1,B_2\in\mathbb{R}^{n\times p}$ and $C=C_1C_2^T$, $C_1,C_2\in\mathbb{R}^{n\times q}$, such that $B_i$, $C_i$ have unit norm and random entries for $i=1,2$. The matrices $A$ and $D$ are as before.

 In Table~\ref{tab1.1} we report the results for different values of $p$, $q$ and $n$.

 \begin{table}[!ht]
 \centering
 \caption{Example \ref{Ex.1}. Results for different values of $p$, $q$ and $n$. \label{tab1.1}}
\begin{tabular}{|r r r r r r r r|}
 \hline
 $n$ & $p$ & $q$ &  Its (inner) & Mem. & Rank($X$)& Rel. Res. & Time (secs) \\
 \hline
 %\hline
\multirow{ 3}{*}{10,000}& 1 & 1 & 13 (6.46) & 160 & 28 & 8.33e-7 & 15.65 \\
 %\hline

 %\hline
&1 & 5 & 6 (6.66) & 624 & 87 & 5.14e-7 & 52.15 \\
 %\hline

 %\hline
& 5 & 10 & 6 (6.00) & 1,560 & 186 & 4.39e-7 & 110.12 \\
 \hline

 %\hline
\multirow{ 3}{*}{22,500}& 1 & 1 & 15 (10.60) & 352 & 26 & 5.18e-7 & 69.19 \\
 %\hline

 %\hline
&1 & 5 & \multicolumn{5}{c|}{convergence not achieved}\\
 %\hline

 %\hline
&5 & 10 & \multicolumn{5}{c|}{convergence not achieved}\\
\hline

\multirow{ 3}{*}{32,400}& 1 & 1 & \multicolumn{5}{c|}{convergence not achieved} \\
 %\hline

 %\hline
&1 & 5 & \multicolumn{5}{c|}{convergence not achieved}\\
 %\hline

 %\hline
&5 & 10 & \multicolumn{5}{c|}{convergence not achieved}\\
\hline

\end{tabular}
 \end{table}

 We notice that for the largest values of $n$, the inexact Newton-Kleinman method does not always achieve the desired accuracy in terms of relative residual norm.
 Indeed, for a certain $k>0$, Algorithm~\ref{T_EKSM} does not manage to solve the $k$-th equation \eqref{Newton_step}
  of the Newton-Kleinman scheme\footnote{Some examples where Algorithm~\ref{T_EKSM} does not converge are reported also in \cite{Dopico2016}.} and we thus stop the process. In Figure~\ref{fig:2} (left) we plot in logarithmic scale the T-Riccati relative residual norm for the case $n=22,500$ and $p=1$, $q=5$. For this example, the residual norm decreases (non monotonically) until Algorithm~\ref{T_EKSM} is no longer able to solve the eighth T-Sylvester equation
  \begin{equation}\label{ex1.eq}
(D-X_7^TB_1B_2^T)\widetilde X_8+\widetilde X_8^T(A-B_1B_2^TX_7)^T=-X_7^TB_1B_2^TX_7-C_1^TC_2.
  \end{equation}
 In particular, in Figure~\ref{fig:2} (right), the relative residual norm (solid line) produced by Algorithm~\ref{T_EKSM} when applied to equation \eqref{ex1.eq} is reported. We can appreciate how the residual norm smoothly decreases in the first 18 iterations and, after an erratic phase, it starts increasing until the 35th iteration when we stop the procedure. In Figure~\ref{fig:2} (right) we also plot the threshold (dashed line) passed to Algorithm~\ref{T_EKSM}, i.e., \mbox{$\eta_7\cdot\|\mathcal{R}_T(X_7)\|_F$}, and we can realize how the relative residual norm gets very close to the desired accuracy without reaching it. A similar behaviour has been observed also for the other tests where the convergence has not been achieved.
 We think it may be interesting to further study the convergence property of Algorithm~\ref{T_EKSM} as also the solution of the T-Riccati equation \eqref{eq.TRiccati} can benefit from this.

  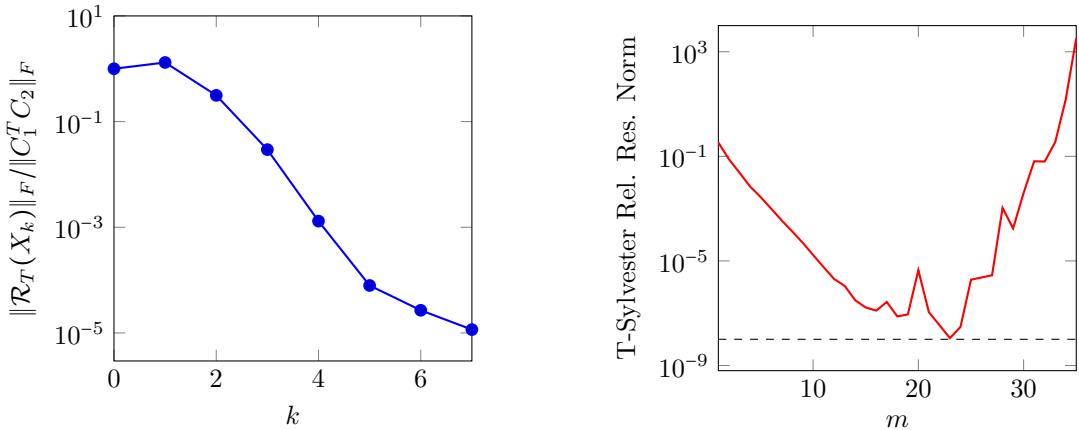
\begin{figure}
  \centering
  \caption{Example~\ref{Ex.1}, $n=22,500$, $p=1$ and $q=5$. Left: $\|\mathcal{R}_T(X_k)\|_F/\|C_1^TC_2\|_F$ for $k=0,\ldots,7$. Right: relative residual norm produced by Algorithm~\ref{T_EKSM} when applied to equation \eqref{ex1.eq} (solid line) and $\eta_7\cdot\|\mathcal{R}_T(X_7)\|_F$ (dashed line).} \label{fig:2}
  \begin{minipage}{.48\linewidth}
  	\centering
  	\begin{tikzpicture}
    \begin{semilogyaxis}[width=0.8\linewidth, height=.27\textheight,
      xlabel = $k$, ylabel = $\|\mathcal{R}_T(X_k)\|_F/\|C_1^TC_2\|_F$,xmin=0, xmax=7, ymax = 1e1]
      \addplot+[thick] table[x index=0, y index=1]  {data2.dat};
       \end{semilogyaxis}
  \end{tikzpicture}

  \end{minipage}~\begin{minipage}{0.480\textwidth}  	
    \bigskip

  \centering

  	\begin{tikzpicture}
    \begin{semilogyaxis}[width=0.8\linewidth, height=.27\textheight,
     xlabel = $m$, ylabel = {T-Sylvester Rel. Res. Norm },xmin=1, xmax=35, ymax = 1e4]
      \addplot+[thick,color=red, mark=red] table[x index=0, y index=1]  {data3.dat};
      \addplot+[mark=,fill=black,color=black,dashed] table[x index=0, y index=2]  {data3.dat};

       \end{semilogyaxis}
  \end{tikzpicture}
  \end{minipage}

\end{figure}
 When the desired accuracy is achieved, the rank of the computed numerical solution $X$ is rather small compared to the problem size $n$, for all the tested values of $p$ and $q$. This suggests that it may be reasonable to investigate in depth the trend of the singular values of the exact solution to \eqref{eq.TRiccati} in order to justify the search for low-rank approximate solutions and the development of low-rank numerical schemes.

 }
\end{num_example}

\begin{num_example}\label{Ex.2}
{\rm
 The second example we consider
 consists in a slight modification of \cite[Example 6.1]{Guo2001}.

 In the small-scale setting we generate a random matrix $R=\mathtt{rand}(2n,2n)\in\mathbb{R}^{2n\times 2n}$ and define $W=\text{diag}(R\mathbf{1})-R$ where $\mathbf{1}=(1,\ldots,1)^T\in\mathbb{R}^{2n}$. Then $A,D\in\mathbb{R}^{n\times n}$ are chosen according to the partition
 $$W=\begin{bmatrix}
      D & M \\
      N & A \\
     \end{bmatrix},
     $$
and $B=-N/\|N\|_F$. We also define the exact solution to \eqref{eq.TRiccati} as
an $n\times n$ matrix $X_{\mathtt{exact}}$ with random entries and unit norm and we compute $C=DX_{\mathtt{exact}}+X_{\mathtt{exact}}^TA-X_{\mathtt{exact}}^TBX_{\mathtt{exact}}$.

The results for different $n$ are collected in Table~\ref{tab2} where we also report the relative error between the computed solution and $X_{\mathtt{exact}}$.

 \begin{table}[!ht]
 \centering
 \caption{Example \ref{Ex.2}. Results for different values of (moderate) $n$. \label{tab2}}
\begin{tabular}{|r| r r r r r|}
 \hline
 & $n$  & Its & Rel. Res. & Err. Rel. & Time (secs) \\
 \hline
 w/o line search & \multirow{ 2}{*}{500} & 3 & 1.06e-14& 7.78e-11 & 10.80 \\
 w/ line search & & 3 & 3.48e-13 & 6.01e-10 & 10.84 \\
\hline
 w/o line search & \multirow{ 2}{*}{1,000} & 3 & 1.49e-14& 9.33e-10 & 78.59  \\
 w/ line search & & 3 & 1.78e-13& 1.45e-9 & 78.99 \\
\hline

\end{tabular}
 \end{table}

 The Newton-Kleinman method with line search performs in a very similar manner with respect to the case where no line search is used. Indeed, in this example, the computed step-size $\lambda_k$ is always close to one, for every $k$.

 For the large-scale setting we have to construct the coefficient matrices in a different way to be able to allocate them.
To this end we compute two sparse matrices $F,G\in\mathbb{R}^{n\times n}$ with random entries via the MATLAB function {\tt sprand}\footnote{The density of the nonzero entries is set to be equal to $1/n$.} and we shift them to ensure their nonsingularity. We thus
 define $D=F+(\rho(F)+1)I$ and
 $A=G+(\rho(G)+20)I$. As in Example~\ref{Ex.1}, we consider low-rank matrices $B=B_1B_2^T$, $B_1,B_2\in\mathbb{R}^{n\times p}$ and $C=C_1C_2^T$, $C_1,C_2\in\mathbb{R}^{n\times q}$ such that $B_i$, $C_i$ have unit norm and random entries for $i=1,2$.

 In Table~\ref{tab2.1} we report the results for different values of $p$, $q$ and $n$.

 \begin{table}[!ht]
 \centering
 \caption{Example \ref{Ex.2}. Results for different values of $p$, $q$ and $n$. \label{tab2.1}}
\begin{tabular}{|r r r r r r r r|}
 \hline
 $n$& $p$ & $q$ &  Its (inner) & Mem. & Rank($X$)& Rel. Res. & Time (secs) \\
 \hline
 %\hline
\multirow{ 3}{*}{10,000}& 1 & 1 &4 (1.5) & 32 & 4 & 6.19e-7 & 0.16 \\
 %\hline
 %\hline
&1 & 5 & 5 (1.8) & 144 & 29 & 1.18e-8 & 1.11 \\
 %\hline
 %\hline
& 5 & 10 & 5 (1.8) & 360 & 60 & 2.35e-9 & 3.27 \\
 \hline

 \multirow{ 3}{*}{50,000}& 1 & 1 &4 (1.5) & 32 & 4 & 6.48e-7 & 0.79 \\
 %\hline
 %\hline
&1 & 5 & 5 (1.8) & 144 & 29 & 1.18e-8 & 5.44 \\
 %\hline
 %\hline
& 5 & 10 & 5 (1.8) & 360 & 60 & 1.19e-9 & 14.88 \\
 \hline

 \multirow{ 3}{*}{100,000}& 1 & 1 &4 (1.5) & 32 & 4 & 6.30e-9 & 1.48 \\
 %\hline
 %\hline
&1 & 5 & 5 (1.8) & 144 & 28 & 1.80e-8 & 11.33 \\
 %\hline
 %\hline
& 5 & 10 & 5 (1.8) & 360 & 60 & 4.71e-10 & 24.49 \\
 \hline

\end{tabular}
 \end{table}

 In this example, we manage to reach the desired accuracy for every value of $p$, $q$ and $n$ we tested. Moreover, the numerical solution turns out to be low-rank in all the experiments we ran.

 We notice that the computational timings in Table~\ref{tab2.1} are several orders of magnitude smaller than the ones reported in Table~\ref{tab1.1} even when the problem dimension, the rank of $B$ and $C$ and the number of outer iterations are very similar. This is mainly due to the following factors. The average numbers of inner iterations in Table~\ref{tab1.1} is larger than the ones
 reported in Table~\ref{tab2.1}. Therefore, even if we solve a similar number of T-Sylvester equations to converge, the ones
 in Example~\ref{Ex.1} require a larger space to be solved leading to an increment in both the memory allocation and the computational efforts. Moreover, each of these inner iterations is more expensive than a single inner iteration with the data of Example~\ref{Ex.2} because of
 the different level of fill in of the coefficient matrices. For instance, for $n=10,000$, the number of nonzero entries of $A$ and $D$ in Example~\ref{Ex.1} is approximately 50,000 while in Example~\ref{Ex.2} is 20,000.
 }
\end{num_example}

%%%%%%%%%%%%%%%%%%%%%%%%%%%%%%%%%%%%%%%%%%%%%%
\section{Conclusions}\label{Conlusions}
By taking inspiration from the rich literature about the algebraic Riccati equation, in this paper we investigated some theoretical and computational aspects of the nonsymmetric T-Riccati equation.
Sufficient conditions for the existence and uniqueness of a minimal nonnegative solution $X_{\min}$ have been provided. We have thoroughly explored the numerical computation of $X_{\min}$ and effective procedures for both small and large problem dimensions have been proposed. The reliability of the derived schemes has been established by showing their convergence to $X_{\min}$ whereas several numerical experiments illustrate their efficiency in terms of both memory requirements and computational time.

In the large-scale setting, low-rank approximate solutions turned out to be accurate in terms of relative residual norm. This suggests that it may be possible to show that the exact solution $X_{\min}$ presents a fast decay in its singular values and this will be the topic of future works. The projection scheme adopted to solve the T-Sylvester equations arising from the Newton-Kleinman iteration failed to converge in some cases so that the solution to the T-Riccati equation
could not be computed. A robust convergence theory for large-scale T-Sylvester equations solvers is still lacking in the literature and we think it
may be a very interesting research topic as also the numerical procedure for T-Riccati equations presented in this paper can benefit from it.

The promising results encourage us to tackle more difficult problems with data coming from real-life applications as the ones discussed in section~\ref{Introduction}.

\section*{Acknowledgments}
This work was inspired by Don Harding (Victoria University, Melbourne, Australia) in the context of the ARC Discovery Grant No DP1801038707 ``New methods for solving large models with rational expectations'' in which the first author serves as a consultant.
Moreover, we wish to thank Froil\'{a}n Dopico and Valeria Simoncini
for providing us with the MATLAB implementations of Algorithm 3.1 in \cite{DeTeran2011} and Algorithm 2 in \cite{Dopico2016}, respectively.

 The second author is member of the Italian INdAM Research group GNCS.
\bibliography{TRiccati}

\end{document}